\documentclass{amsart}
 \usepackage[foot]{amsaddr}

\usepackage{amsmath}
\usepackage{amssymb}
\usepackage{mathtools}
\usepackage{cite}
\usepackage[alphabetic]{amsrefs}
\usepackage{stmaryrd}
\usepackage[mathscr]{euscript}
\usepackage{amsthm}
\usepackage{enumerate}
\usepackage{enumitem}
\usepackage{hyperref}
\usepackage{caption}
\usepackage{xcolor}
\usepackage{comment}
\usepackage{appendix}
\usepackage{subcaption}

\newtheorem{thm}{Theorem}[section]
\newtheorem{cor}{Corollary}[section]
\newtheorem{lem}{Lemma}[section]
\newtheorem{prop}{Proposition}[section]

\newtheorem{defn}{Definition}[section]
\newtheorem{rem}{Remark}[section]

\begin{document}
\title{Stability of Convex Spheres}

\author{Davi M\'aximo and Hunter Stufflebeam}

\address{University of Pennsylvania, Department of Mathematics, Philadelphia, PA, USA.}

\email{dmaxim@math.upenn.edu or hstuff@sas.upenn.edu}

\begin{abstract}
We prove that strictly convex $2$-spheres, all of whose simple closed geodesics are close in length to $2\pi$, are $C^0$ Cheeger-Gromov close to the round sphere. 
\end{abstract}
\maketitle

\section{Introduction}

A famous theorem of V. Toponogov states that the length of any simple, closed geodesic on a closed surface $M$ of curvature $K\geqslant 1$ is bounded above by $2\pi$. Moreover, $M$ has a simple closed geodesic of length $2\pi$ if and only if $M$ is isometric to the standard round $\mathbb S^2$. In this paper, we investigate the question of stability for such a result. Namely, does the presence of long simple closed geodesics imply global closeness to the round metric in some sense?

This natural question was the original motivation for the work of the second named author in \cite{HS} regarding convex disks with large boundary. There it was observed that there are spherically symmetric metrics on the sphere with $K\geqslant 1$ having simple closed geodesics of length arbitrarily close to $2\pi$ which are far from the round metric. These football-like spheres illustrate that having \emph{some} simple closed geodesic of nearly maximal length is not enough to get global closeness to the round sphere. Nevertheless, this class of examples suggests various ways to adapt the question so as to potentially obtain stability. In this paper, we explore the situation where the width of the sphere, which is realized as the \emph{shortest} length of a closed geodesic, is close to $2\pi$. Our main result is that this condition does indeed yield stability, and even for a relatively strong topology.

In what follows, given a 2-manifold $(M,g)$ we use the notation $K_g$ for the Gaussian curvature and $L_{g}$ for the length functional of $g$. The main invariant of interest to us, the {\it width} of $g$, is denoted by \[\ell_g:=\inf\{L_g(\gamma): \gamma \text{ is a closed geodesic on } M\}.\] 
Our main result is the following stability statement for $\ell_g$ for metrics with $K_g\geq 1$: (Here and throughout the paper, $d_{C^0}$ denotes the $C^0$ Cheeger-Gromov distance between Riemannian manifolds)
\begin{thm}
For any $\delta>0$, there exists an $\varepsilon=\varepsilon(\delta)>0$ such that the following holds. Let $(M, g)$ be a closed Riemannian surface of curvature $K_g\geqslant 1$, and suppose that $\ell_g\geqslant 2\pi-\varepsilon.$ Then \[d_{C^0}\left((M,g),(\mathbb S^2, g_{rd})\right)\leqslant\delta.\]
\end{thm}
A few remarks about the invariant $\ell_g$ are in order. A classical min-max theorem of G. D. Birkhoff \cite{BirkhoffGeodesic}, using 1-parameter family of curves, states that every Riemannian metric on 2-sphere admits a closed geodesic. Thus, for any smooth metric $g$, $\ell_g$ is well-defined. Birkhoff's result was subsequently improved by L. Lyusternik-L. Schnirelmann \cite{LyusternikSchnirelmann}, showing the existence of at least 3 closed geodesics on any smooth 2-sphere (see also \cite{GraysonCurveShortening1989} for an argument using curve shortening flow). Later, using geometric measure theory and results from F. J. Almgren \cite{AlmgrenHomotopyGroups}, J. Pitts \cites{PittsThesis, Pitts1974} developed a new existence theory, which became known in the literature as Almgren-Pitts min-max theory, and in this framework (see Section 2.2 for more details) E. Calabi-J. Cao \cite{CalabiCao} showed that, on spheres of non-negative curvature, $\ell_g$ is attained by a \emph{simple} closed geodesic.

\begin{rem}
    The statement of our Theorem 1.1 is similar to the main result proved by R. Bamler-D. M\'aximo in \cite{BMrf} about the stability of the area estimate (appearing as Theorem 1.1) of F. Coda-Marques-A.Neves in \cite{MNSphere} for min-max minimal two-spheres in three-manifolds. Their proof, which also takes place under the assumption of positive sectional curvature, nonetheless uses altogether different methods. 
\end{rem} 

Our proof proceeds in two main steps by way of contradiction, and via a result of T. Yamaguchi in \cite{Yamaguchi}, it suffices to establish the result using the Gromov-Hausdorff distance $d_{GH}$ in place of $d_{C^0}$. Given a hypothetical sequence of such spheres $M_k$ with $\ell_{g_k}:=\ell_k\nearrow 2\pi$ which remain bounded away from round $\mathbb S^2$ in $d_{GH}$, we seek to extract a subsequence converging to the round sphere. In the first step, we use synthetic approaches from Alexandrov geometry to obtain a subconverging sequence, whose limit is a two-dimensional Alexandrov space of $\mathrm{curv}\geqslant 1$. Using a non-collapsing estimate of Croke \cite{Croke}, the Perelman Stability Theorem, and K. Grove-P. Petersen's resolution of the Lytchak problem \cite{GrovePetersenLens} (see Section \ref{2.1}), we identify this limit space as a gluing of two round spherical lunes, also known as \emph{Alexandrov lenses}. 

In the second step, we further identify the metric of the limit as that of the round sphere, forcing the desired contradiction. The main tool in this part is the Almgren-Pitts min-max method for one-cycles. Assuming for the sake of contradiction that the limit metric were not round, we construct a sweepout of the limit with width smaller than $2\pi$. We then pull this sweepout back to the sequence spheres $M_k$ using bi-Lipschitz maps to obtain sweepouts of the $M_k$ with boundedly small widths. By then applying the Almgren-Pitts  theory, we would be able to violate the definition of $\ell_k$ on $M_k$ far enough along in the sequence, thereby identifying the limit as the round sphere and establishing the main theorem.

\subsection{Notation and Conventions} 

Throughout this paper, we will often work on metric spaces which underlie Riemannian manifolds. For that reason, we will reference metric quantities by acknowledging the overlying metric. For example, we will write $d_g$ for the distance function deriving from the Riemannian metric $g$. Similarly, for example, the $2$-dimensional Hausdorff measure deriving from the metric structure of $(M,g)$ will be denoted by $\mathrm{Area}_{g}$. In case a measure is omitted from an integral, it is understood that the implied measure is the standard volume measure on the underlying space. Finally, we will follow tradition in using notation such as $\Psi=\Psi(x)=\Psi(x|a_1,a_2,\ldots)$ to denote a non-negative function, which may change from line to line, depending on a variable $x$ and any number of parameters $a_i$ with the property that if the $a_i$ are all held fixed, $\Psi\searrow 0$ as $x\to 0$. 

\subsection{Acknowledgements}
The authors are indebted to the anonymous reviewers whose comments significantly improved our manuscript. D.M. is thankful to Ruobing Zhang for helpful discussions and to Alexander Lytchak for useful references in Alexander Geometry. H.S is thankful to their advisor D.M. for his constant support and encouragement. 

\section{Preliminaries} 

\subsection{Metric Geometry and the Convergence of Alexandrov Spaces}\label{2.1}

We will assume some familiarity with the basics of metric geometry, and especially Alexandrov spaces. For some good general references, see the book \cite{BBI} and the seminal paper \cite{BGP}. In particular, we will assume that the reader is familiar with the definition of Alexandrov spaces ($i.e.$ complete length spaces with a metric notion of curvature being bounded from below), as well as some technical concepts such as $(n,\delta)$-strainers, spaces of directions, and singular points. For the definitions and a quick recap of these latter concepts, the first section of \cite{Yamaguchi} is a great source. In what follows, we will denote the class of $n$-dimensional Alexandrov spaces of $\mathrm{curv}\geqslant\kappa$ by $\mathrm{Alex}^n(\kappa)$. 

We will utilize the following important result due to Grove-Petersen \cite{GrovePetersenLens} about the structure of Alexandrov spaces with maximal boundary volume, which answered the following questions of A. Lytchak: If an Alexandrov space with boundary has $\mathrm{curv}\geqslant 1$, how big can the boundary be? And, if there are extremizers, what are they? First recall that if $(E, d_E)$ is a metric space with diameter $\leqslant\pi$, then the \emph{spherical suspension of $E$} is the metric space
\[\Sigma_1 E:=\left(\left(\left[-\frac{\pi}{2}, \frac{\pi}{2}\right]\times E\right)/\sim, d\right)\] where $\sim$ is the equivalence relation which collapses $\{-\pi/2\}\times E$ and $\{\pi/2\}\times E$ to points $p_-$ and $p_+$, respectively, leaves all other points alone, and where $d$ is the metric 
\begin{equation*}
 d(p_1, p_2)=\arccos\left(\sin(r_1)\sin(r_2)+\cos(r_1)\cos(r_2)\cos d_E(e_1,e_2)\right)
 \end{equation*}
with $p_i=\overline{(r_i, e_i)}\in (\left[-\pi/2, \pi/2\right]\times E)/\sim$. A key example for us is a so-called \emph{Alexandrov lens}, denoted by $L^n_\alpha$, which is the intersection of two round spherical $n-$hemispheres making an angle of $\alpha\in (0,\pi]$. In the two dimensional case which we are concerned about, we can equivalently write $L^2_\alpha=\Sigma_1 [0,\alpha]$---the spherical suspension of an interval of length $\leqslant\pi$ (see Figure \ref{fig:lensdetail}). These special Alexandrov spaces are important because of the following comparison (due to A. Petrunin) and rigidity theorem (due to Grove-Petersen): 

\begin{thm}[Boundary Volume Comparison and Rigidity Theorem cf. \cite{Petruninsemiconcave} and \cite{GrovePetersenLens}]\label{GrovePetersen}
Suppose $X\in\mathrm{Alex}^n(1)$ with $\partial X\neq\emptyset$. Then \[\mathrm{vol}^{n-1}(\partial X)\leqslant \mathrm{vol}^{n-1}(\mathbb S^{n-1}, g_{rd}),\] with equality iff $X$ is isometric to an Alexandrov lens $L^n_\alpha$ with $0<\alpha\leqslant\pi$.    \end{thm}

We will also need to deal with sequences of Alexandrov spaces and their convergence. To that end, we first recall for the readers convenience the notion of Gromov-Hausdorff convergence, first introduced by D. Edwards in \cite{DEdwards} and popularized by M. Gromov in \cite{GromovStructures} and \cite{GromovPolynomial}. The Gromov-Hausdorff distance between two compact metric spaces $(X_1,d_1)$ and $(X_2,d_2)$ can be defined by
\[
d_{\mathrm{GH}}((X_1,d_1),(X_2,d_2))=\inf_{Z} \{d_H^Z(\phi_1(X_1),\phi_2(X_2))\}
\]
where the infimum is taken over all complete metric spaces $(Z,d^Z)$ and all distance preserving maps $\phi_i:X_i\to Z$, and where $d_H^Z$ denotes the standard Hausdorff distance between two compact subsets of $(Z,d^Z)$: For any compact  $X, Y\subset Z$,
\[d_H^Z(X,Y)=\inf\{r>0: X\subset B_r(Y)\text{ and } Y\subset B_r(X)\}\] with $B_r(\cdot)$ denoting the $r$ neighborhood of a subset of $(Z,d^Z)$. We say that a metric spaces $(X_j,d_j)$ converge in the $\mathrm{GH}$-sense to a metric space $(X_\infty,d_\infty)$ if
\[
d_{\mathrm{GH}}((X_j,d_j),(X_\infty,d_\infty)) \to 0.
\] An equivalent, and useful, way of the framing Gromov-Hausdorff distance is via the following notion: a (potentially noncontinuous) map of metric spaces $f:(X,d_X)\to (Y, d_Y)$ is called an $\varepsilon$-\emph{Gromov-Hausdorff approximation} ($\varepsilon$-GH approximation for short) provided that $f(X)\subset Y$ is an $\varepsilon$-net and \[\mathrm{dis}(f):=\sup_{x,y\in X}\left\vert d_X(x,y)-d_Y(f(x),f(y))\right\vert<\varepsilon.\] It is a nice exercise to show that $d_{GH}(X,Y)<\varepsilon$ implies the existence of a $2\varepsilon$-GH approximation, and the existence of an $\varepsilon$-GH approximation implies $d_{GH}(X,Y)<2\varepsilon$ (cf. \cite{BBI} Chapter 7). Besides $\varepsilon$-GH approximations, another class of maps which we will deal with is the class of so-called $\varepsilon$-\emph{almost isometries}. Such a mapping $f:(M,d_M)\to (X, d_X)$ is an $\varepsilon$-GH approximation satisfying the further condition that for every $x,y\in M$ \[\left\vert\frac{d_X(f(x),f(y))}{d_M(x,y)}-1\right\vert<\varepsilon.\]

There are three primary results about the Gromov-Hausdorff convergence of Alexandrov spaces which will be useful to us, the first of which is the following result allowing us to produce limits from a sequence: 

\begin{thm}[Gromov Compactness cf. \cite{BBI} Theorem 10.7.2]\label{GH} 
Fix $n\in\mathbb Z^{\geqslant 0}$, $\kappa\in\mathbb R$, and $D>0$. Let $\mathcal M(n,\kappa, D)$ denote the class of Alexandrov spaces of Hausdorff dimension $\leqslant n$, $\mathrm{curv}\geqslant\kappa$, and $\mathrm{diam}\leqslant D$. This class, endowed with the Gromov-Hausdorff topology, is compact.
\end{thm} Of particular relevance to us is the consequence that the class of Riemannian manifolds of $\dim=2$ with $\mathrm{sec}\geqslant\kappa>0$ is pre-compact in the Gromov-Hausdorff topology, and that the limit points of such sequences lie in the class of Alexandrov spaces of $\mathrm{curv}\geqslant\kappa$ and $\dim\leqslant 2$.

 The second important result we will need from metric geometry is the following celebrated result of G. Perelman: \begin{thm}[Perelman Stability Theorem cf. \cite{BBI} Theorem 10.10.5 and \cite{Kapovitch}]\label{PerelmanStability} 
Fix $\kappa\in\mathbb R$ and a compact Alexandrov space $X\in\mathrm{Alex}^n(\kappa)$. Then for any $\delta>0$, there exists an $\varepsilon=\varepsilon(X,\delta)>0$ such that if $Y\in\mathrm{Alex}^n(\kappa)$ has $d_{GH}(X,Y)<\varepsilon,$ then $X$ and $Y$ are homeomorphic via a $\delta$-GH approximation.  
\end{thm}
In particular, for a GH-converging, non-collapsing sequence $X_i\to X$ in $\mathrm{Alex}^n(\kappa)$ (i.e. the dimension of the limit $X$ is also $n$), for all large $i$ there are homeomorphisms $f_i:X_i\to X$ which have $\mathrm{dis}(f_i)\to0$. The third result, due to T. Yamaguchi, upgrades these homeomorphisms to $\varepsilon$-almost isometries when limit space $X$ is known to have only \emph{mild} singularities. To be precise, given a point $p\in X$ and $\delta>0$ we define the $\delta$-strain radius of $X$ at $p$ by \[\delta\mathrm{str.rad}(p)=\sup\{r>0: \text{there exists an $(n,\delta)$-strainer based at $p$ with length $r$}\}\] and the $\delta$-strain radius of $X$ by $\delta\mathrm{str.rad}(X)=\inf_{p\in X}\delta\mathrm{str.rad}(p).$ As explained in Section 1 of \cite{Yamaguchi}, the $\delta$-strain radius is effectively a non-smooth analogue of the injectivity radius. For our purposes, having \emph{mild} singularities amounts to asking for a lower bound on $\delta\mathrm{str.rad}(X)$:\footnote{Recall (see for example \cite{BGP} Theorem 10.4) that there is a $\delta_n>0$ such that any point $p\in X\in\mathrm{Alex}^n(\kappa)$ with $\delta\mathrm{str.rad}(p)\geqslant\mu_0>0$, for some $0\leqslant\delta\leqslant\delta_n$, is in the domain $B_\sigma(p)$, $\sigma<<\mu_0$, of a coordinate chart into $\mathbb R^n$ which is a $\Psi(\delta, \sigma)$-almost isometry.} 

\begin{thm}[Theorem 0.2, Corollary 0.4, and Remark 4.20 \cite{Yamaguchi}]\label{Yam}
Fix a positive integer $n$ and a small $\mu_0>0$. Then there exist $\delta_n,\varepsilon_n(\mu_0)>0$ and a function $\tau=\tau(\delta, \varepsilon\vert n, \mu_0)>0$ with $\tau(\delta, \varepsilon\vert n, \mu_0)\searrow 0$ as $\delta, \varepsilon\searrow 0$ such that the following holds:

Let $M, X\in\mathrm{Alex}^n(-1)$ with $\delta\mathrm{str. rad}(X)>\mu_0$ for some $0<\delta\leqslant\delta_n$, and assume that $d_{GH}(M,X)<\varepsilon\leqslant\varepsilon_n(\mu_0)$. Then there exists a $\tau(\delta,\varepsilon)$-almost-isometry $f:M\to X$. Moreover, in case $M$ and $X$ have $C^1$ differentiable structures with $C^1$ distance functions, $f$ can be taken to be $C^1$ as well. 
\end{thm}

We will need to apply a slightly more general result in our subsequent work, which nonetheless follows directly from the inherently local argument of Yamaguchi as written (see also the proof of Corollary 0.6 in \cite{Yamaguchi}). In particular, we need a localized version of this result to give such a $\tau$-almost-isometry on a subregion of $X$ where the $\delta\mathrm{str.rad}$ estimate holds uniformly, while $X$ might have more serious singularities elsewhere. 

\begin{cor}[Local Yamaguchi Theorem \cite{Yamaguchi}]\label{locYam}
Suppose, in the situation of Theorem \ref{Yam} above, that the estimate $\delta\mathrm{str. rad}(x)>\mu_0$ is only assumed to hold for $x$ in some region $\Omega\subset X$. Then there exists a homeomorphic subregion $\tilde\Omega\subset M$ and a $\tau(\delta,\varepsilon)$-almost-isometry $f:\tilde \Omega\to \Omega$ (with their restricted metrics).
\end{cor}

Lastly, concerning preliminary geometric convergence results, we will use the following proposition allowing us to reduce proving convergence in the $C^0$-Cheeger-Gromov sense to proving convergence in $GH$. 

\begin{prop}\label{GHtoC0}
    Suppose that $(M_k, g_k)$ and $(X, g)$ in the statement of Theorem \ref{Yam} are all smooth, closed Riemannian $n$-manifolds with $\mathrm{sec}(M_k)\geqslant\kappa$ for all $k$. Then $(M_k, g_k)$ converges to $(X,g)$ in the $C^0$-Cheeger-Gromov sense as well.  
\end{prop}

\begin{proof}
    By Theorem \ref{Yam}, for all large $k\geqslant 1$ we can obtain $C^1$ diffeomorphisms $f_k: (X,g)\to (M_k, g_k)$ which satisfy the Lipschitz estimates \[\left\vert\frac{d_{g_k}(f_k(x),f_k(y))}{d_g(x,y)}-1\right\vert<\Psi(k^{-1})\] for every $x,y\in X$. Since the $f_k$ are $C^1$, this implies that \[1-\Psi(k^{-1})\leqslant\|Df_k\|_{g,g_k}\leqslant 1+\Psi(k^{-1}),\] where $\|-\|_{g,g_k}$ denotes the operator norm with respect to the metrics $g$ and $g_k$. But then for any $v\in TX$, we have that \[(1-\Psi(k^{-1}))|v|_g\leqslant|v|_{f_k^*g_k}=|Df_k(v)|_{g_k}\leqslant (1+\Psi(k^{-1}))|v|_g\] which proves the claim. 
\end{proof}

\subsection{Almgren-Pitts Min-Max Theory for Integral Cycles}

Let $(M,g)$ be a closed Riemannian manifold, and $\Lambda^1(M)$ the space of smooth one-forms on $M$. Given an oriented one-dimensional rectifiable set $\Gamma\subset M$ there is a corresponding linear functional on $\Lambda^1(M)$, called a \emph{rectifiable one-current}, defined by \[T(\phi)=\int_\Gamma \phi.\]The \emph{boundary of $T$}, denoted by $\partial T$, is the linear functional on $C^\infty(M)$ defined by \[\partial T(f)=T(df).\] If $\partial T$ is the trivial current, which we denote in any dimension by $\llbracket 0\rrbracket$, then $T$ is said to be \emph{closed}. In short, we call a closed and rectifiable one-current a \emph{one-cycle}. The \emph{group of one-cycles in $M$ with $\mathbb Z$ coefficients} is then given by \[\mathcal {Z}_1(M, \mathbb Z)=\left\{\sum_{i=1}^\infty a_iT_i~\colon a_i\in\mathbb Z\text{~and~} T_i\text{~is a rectifiable one-cycle}\right\}.\]

We endow $\mathcal {Z}_1(M, \mathbb Z)$ with the \emph{weak-topology}, by defining convergence \[T_i\rightharpoonup T\] as the condition that \[\lim_{i\to\infty}T_i(\phi)=T(\phi)\] for every fixed $\phi\in\Lambda^1(M)$. We may also endow $\mathcal {Z}_1(M, \mathbb Z)$ with the \emph{mass norm}, which is simply the operator norm of a current defined with respect to the induced Riemannian norm on $\Lambda^1(M)$: \[\mathbb{M}(T):=\sup_{\substack{\phi\in \Lambda^1(M) \\ |\phi|_g\leqslant 1}} T(\phi).\] When the underlying rectifiable set $\Gamma$ of $T$ is a Lipschitz curve with finitely many self-intersections, it turns out that \[\mathbb{M}(T)=L_g(\Gamma).\]

In this paper, we will use integral one-cycles to construct one-parameter sweepouts of various manifolds.

\begin{defn}[Sweepout by Integral One-Cycles cf. \cite{CalabiCao}]\label{rough}
Let $(\mathbb S^2, g)$ be a Riemannian sphere. A 1-sweepout of $(\mathbb S^2, g)$ by integral one-cycles is a continuous non-contractible path \[\sigma\in \pi_1(\mathcal Z_1((\mathbb S^2, g), \mathbb Z), \llbracket 0\rrbracket).\] That is,
\begin{itemize}
\item For each $t\in [-1,1]$, $(\sigma)_t:=\sigma(t)\in \mathcal Z_1((\mathbb S^2, g), \mathbb Z)$;
\item The map $t\mapsto (\sigma)_t$ is continuous with respect to the weak topology on $\mathcal Z_1((\mathbb S^2, g), \mathbb Z)$;
\item $\sigma(-1)$ and $\sigma(1)$ are induced by point curves;
\item The path $\sigma$ has $[\sigma]\neq 0$ in $\pi_1(\mathcal Z_1((\mathbb S^2, g), \mathbb Z), \llbracket 0\rrbracket)$.
\end{itemize}
\end{defn}

One then defines the following Almgren-Pitts min-max width:

\begin{defn}[Min-Max Width cf. \cite{CalabiCao}]\label{width}
For a given Riemannian sphere $(\mathbb S^2, g)$, the min-max value $W(\mathbb S^2, g)$ for the mass norm $\mathbb M$ is given by \[W(\mathbb S^2, g):=\inf_{[\sigma]\neq 0}\sup_{t\in [-1,1]}\mathbb{M}(\sigma(t)).\]
\end{defn}

The following theorem, originally due to Almgren-Pitts, is one of the two main min-max results which we will utilize:

\begin{thm}[Almgren-Pitts  cf. \cites{Pitts1974,PittsThesis, CalabiCao}]\label{AlmgrenPitts}
For a given Riemannian sphere $(\mathbb S^2, g)$, the min-max value $W(\mathbb S^2, g)$ is realized by a nontrivial closed geodesic $\gamma_0$. Moreover, if $\gamma_0$ is the shortest closed geodesic on $(\mathbb S^2, g)$, then $\gamma_0$ is either simple or a figure eight. 
\end{thm}

We will pair this with the following theorem of Calabi-Cao:

\begin{thm}[Calabi-Cao \cite{CalabiCao}]\label{CalabiCao} 
If $g$ is a $C^3$ metric on $\mathbb S^2$ with non-negative curvature, then any nontrivial closed geodesic of shortest length is simple. 
\end{thm}

\begin{rem}
By considering sweepouts that depend upon $p$ parameters, $p>1$, one can generalize Definition \eqref{width} and obtain other min-max invariants called $p$-widths, which are a nonlinear analog of the spectrum of the Laplace–Beltrami operator of  $g$ (cf. \cite{GromovSpectrum1988}) and even obey an analogous asymptotic Weyl Law (cf. \cite{WeylLawVolume}). Following Theorem \ref{AlmgrenPitts}, Pitts showed that the higher parameter widths
are attained by the length of geodesic nets (cf. \cite{Pitts1974}). This was recently improved by work of O. Chodosh-C. Mantoulidis in \cite{ChodoshMantoulidis2023}, who showed that the $p$-widths of any closed Riemannian two-manifold correspond to a union of closed immersed geodesics.
\end{rem}

\begin{rem}
Although we are only concerned with two-manifolds in this paper, the min-max theory of Almgren-Pitts goes far beyond and the rich body of work studying it is impossible to adequately survey here. We highlight the solution of S.T. Yau's conjecture in \cite{YauSeminar} about the existence of infinitely many embedded minimal surfaces: the work of Marques-Neves \cite{MNInfinite} established the existence of infinitely many distinct, closed, embedded, smooth minimal hypersurfaces in any closed $n$-dimensional Riemannian manifold, $3\leqslant n\leqslant 7$, without stable minimal hypersurfaces; and the work of  A. Song in \cite{SongInfinite} analyzed the remaining cases, when stable minimal hypersurfaces are present. In addition, Marques-Neves \cite{MNWillmore} also used min-max theory to resolve the famous \emph{Willmore Conjecture}, and in \cite{MultOne}, X. Zhou proved that the minimal hypersurface produced by the Almgren-Pitts theory in the same ambient setting as \cite{MNInfinite} has, in the orientable case, multiplicity one. Lastly, we spotlight the work of A. Song in \cite{SongEmbedded}, wherein a higher dimensional analogue of Theorem \ref{CalabiCao} is proven: there exists a least area closed minimal hypersurface in any closed $n$-dimensional Riemannian ambient manifold with $3\leqslant n\leqslant 7$, and any such hypersurface is embedded.
\end{rem}

For technical reasons pertaining to the possibility of working on a singular limit, we will need to construct our sweepouts for the sequence manifolds $(M_k, g_k)$ out of two pieces. The large parts of the sweepouts will come from a foliation of the region of the limit space which is bounded away from the singular set. The remaining un-swept portions of the $M_k$, which consist of small disks, are then swept-out separately using the following result: 

\begin{thm}[Liokumovich-Nabutovsky-Rotman Theorem 1.1 \cite{LNR}, see also Papasoglu Theorem 2.2,\cite{Papasoglu}]\label{LNR}
For any two-dimensional Riemannian disk $(D,g)$ and any point $p\in\partial D$, there exists a Lipschitz homotopy $\gamma_t$ of loops based at $p$ with $\gamma_0=\partial D$ and $\gamma_1=\{p\}$ such that \[L_g(\gamma_t)\leqslant 2L_g(\partial D)+686\sqrt{\mathrm{Area_g}(D)}+2\mathrm{diam}_g(D)\] for every $t\in [0,1]$.      
\end{thm}

Although stated for Riemannian disks with smooth boundary, the argument of \cite{LNR} can be adapted to the case of a Riemannian disks with Lipschitz boundaries as we will require. Indeed, the key tool of \cite{LNR}, the so-called Besicovitch Lemma, readily extends to such disks (see Exercise 5.6.10 (2) of \cite{BBI}). The argument then proceeds as in \cite{LNR} verbatim. The two partial sweepouts thus obtained will then be glued together to give sweepouts with good estimates for the min-max process. 

\section{Proof of the Main Theorem}

Let us restate the main theorem for convenience: 

\begin{thm}
For any $\delta>0$, there exists an $\varepsilon=\varepsilon(\delta)>0$ such that the following holds. Let $(M, g)$ be a closed Riemannian surface of curvature $K_g\geqslant 1$, and suppose that $\ell_g\geqslant 2\pi-\varepsilon.$ Then \[d_{C^0}\left((M,g),(\mathbb S^2, g_{rd})\right)\leqslant\delta.\]
\end{thm}

\begin{proof}
We will first prove the result with $d_{C^0}$ replaced by $d_{GH}$. If that result were false, there would exist some fixed $\delta_0>0$ and for each $k\geqslant 1$ an $(M_k, g_k)$ as in the statement with $\ell_{k}:=\ell_{g_k}\nearrow 2\pi$ satisfying \[d_{GH}\left((M_k,g_k),(\mathbb S^2, g_{rd})\right)\geqslant\delta_0>0.\] The proof now proceeds in two parts. In the first we show that the $(M_k, g_k)$ Gromov-Hausdorff subconverge to a two-dimensional Alexandrov space of $\mathrm{curv}\geqslant 1$, which we identify as having the structure of two Alexandrov lenses glued together at their boundaries. In the second part, using handmade sweepouts and min-max theory, we show that this limit space must be the round sphere, forcing a contradiction and establishing our main result.

By Gromov's Compactness Theorem \ref{GH}, there is an Alexandrov space $(X,d)$ of $\mathrm{curv}\geqslant 1$ and $\mathrm{dim}\leqslant 2$ such that, up to a not-relabeled subsequence, $(M_k, g_k)\to (X,d)$ in the Gromov-Hausdorff sense. By the results of Croke and others\footnote{See \cites{BeachRotman, Croke, Rotman, NabutovskyRotman, Sabourau}}, we have that \begin{equation*}\label{Croke}
    \ell_k^2\leqslant C\mathrm{Area}_{g_k}(M_k)
\end{equation*}for a universal constant $C$, so the sequence of $(M_k,g_k)$ is non-collapsing and the limit $(X,d)$ has Hausdorff dimension $2$. Perelman's Stability Theorem \ref{PerelmanStability} then yields, for all large $k$, homeomorphisms $\phi_k :(X,d)\to (M_k,d_{g_k})$ with $\mathrm{dis}(\phi_k)\to 0$. In particular, $X$ is a topological two-sphere.

\begin{rem}
    Recalling the work of Colding in \cite{ColdingShape}, one might initially hope that a volume non-collapsing estimate such as Croke's above might prove volume convergence to the round sphere. Unfortunately, however, the best dimensional constant $C$ in this inequality is still unknown, although for Riemannian two-spheres it is conjectured (\cite{Croke}) to be $2\sqrt{3}$ with the Calabi-Croke sphere as an extremizer. The current best known constant seems to be $C=32$, due to \cite{Rotman} (interestingly, the same constant works for a non-compact version of the problem cf. \cite{BeachRotman}). In any case, this estimate cannot prove the volume convergence we would like and we must proceed in a different way. Thus, we stress that here we only use Croke's estimate to prove non-collapsing, to ensure that the dimension of the limit $(X,d)$ is 2. 
\end{rem} 

\begin{rem}
Similarly, it is an attractive idea to try showing that the limit $(X,d)$ has maximal diameter $\pi$, from which it would follow that the limit has the structure of a \emph{spherical suspension}. Indeed, in \cite{GromovFilling} Gromov asked if there are dimensional constants $c(n)$ such that the shortest length of a simple closed geodesic $\ell_g$ on a closed $n-$dimensional manifold $(M^n, g)$ satisfies $\ell_g\leqslant c(n)\mathrm{diam}_g(M^n)$. If $\pi_1(M)\neq 0$, then it is known via a curve shortening argument (see section 1.1.1 of \cite{RotmanGeodesicNet2006}) that $\ell_g\leqslant 2\mathrm{diam}_g(M)$. In the case of $M^n\simeq\mathbb S^2$, Croke and others\footnote{See \cites{Croke,MaedaGeodesic,NabutovskyRotmanShortestGeodesic,SabourauGeodesic,AdelsteinPallete,RotmanLengthofShortestGeodesic2005}.} were able to show that $\ell_k\leqslant c(2)\mathrm{diam}_g(\mathbb S^2)$, with the best known constant in the non-negatively curved case being $c(2)=3$, due to Adelstein-Pallete \cite{AdelsteinPallete} and in general $c(2)=4$, due to Rotman \cite{RotmanLengthofShortestGeodesic2005}. In the end, the conjecture that the best universal constant $c(2)=2$ was disproved via Zoll metric counterexamples of Balacheff-Croke-Katz in \cite{BalacheffCrokeKatz}, satisfying $\ell_g>2\mathrm{diam}_g(\mathbb S^2)$. In any case, it is impossible for us to extract sharp enough diameter information on the $(M_k ,g_k)$ from the invariant $\ell_k\nearrow 2\pi$ alone to show that $(X,d)$ has maximal diameter $\pi$.
\end{rem}

Passing to the level of metric spaces, we may therefore work directly on a fixed background $\mathbb S^2$ by replacing $(X,d)$ with $(\mathbb S^2, d)$, and $d_{g_k}$ on $M_k$ with $d_k:=\phi^*_k d_{g_k}$ on $\mathbb S^2$. Because $\mathrm{dis}(\phi_k)\to 0$, we have that the $d_k$ converge uniformly to $d$ on $\mathbb S^2\times\mathbb S^2$. By Calabi-Cao's Min-Max Theorem \ref{CalabiCao}, each $(M_k, g_k)$ contains a simple closed geodesic $\gamma_k$ of length $\ell_k\nearrow 2\pi$, and it divides $M_k$ into two topological disks. By Croke's estimate \ref{Croke}, at least one of these disks will have $g_k$ area at least $\mathrm{Area}_{g_k}(M_k)/2\geqslant \ell_k^2/2C\geqslant c_0>0$ for all large $k\geqslant 1$, and we call this disk $U_k$. Because $\gamma_k=\partial U_k$ is geodesic, we in fact have that $(\overline{U_k}, d_{g_k})\in\mathrm{Alex}^2(1)$ for every $k\geqslant 1$. Pulling these Alexandrov sub-disks back to the limit space $(\mathbb S^2,d)$ via the homeomorphisms $\phi_k$, we obtain open disks $\mathcal U_k:=\phi_k^{-1}(U_k)$. Choosing a further subsequence, we can assume by Blaschke's Theorem (cf. Theorem 7.3.8 in \cite{BBI}) that there is a closed subset $\overline{\mathcal U}$ to which the closed disks $\overline{\mathcal U_k}$ converge in the $d$-Hausdorff distance on $(\mathbb S^2,d)$. 

We now claim that $\overline{\mathcal U}$ is in fact a closed topological disk, and that $(\overline{\mathcal U}, d)\in\mathrm{Alex}^2(1)$. First observe that the disks $(\overline{\mathcal U_k}, d_k)$ are themselves in $\mathrm{Alex}^2(1)$, since they are isometric to the disks $(\overline{U_k}, d_{g_k})$. Moreover, $(\overline{\mathcal U_k}, d)$, with the restriction of $d$, has $d_{GH}((\overline{\mathcal U_k}, d_k), (\overline{\mathcal U_k}, d))\leqslant\Psi(k^{-1})$ on account of the uniform convergence of $d_k\to d$ on $(\mathbb S^2,d)$. Therefore, we conclude that $(\overline{U_k}, d_{g_k})\to (\overline{\mathcal U}, d)$ in the Gromov-Hausdorff sense, proving that $(\overline{\mathcal U}, d)$ is an Alexandrov space of $\mathrm{curv}\geqslant 1$ and $\mathrm{dim}\leqslant 2$. Since $(\overline{U_k}, d_{g_k})$ is volume non-collapsing, we conclude that $(\overline{\mathcal U}, d)\in\mathrm{Alex}^2(1)$, and Perelman Stability implies that it is a topological disk. In particular, the boundary $\gamma$ of $\overline{\mathcal U}$ is a simple closed curve dividing $\mathbb S^2$ into two topological 2-disks. Moreover, $\gamma$ has $d$-length $2\pi$. Indeed, $\gamma=\partial\mathcal U$ arises as the Gromov-Hausdorff limit of the curves $\gamma_k=\partial U_k$ in each $M_k$ (cf. Theorem 1.2 in \cite{PetruninApplications}--recall that the boundary of an Alexandrov space is a so-called \emph{extremal set}, about which the cited theorem applies). Since this limit is non-collapsed, it follows from weak convergence of the $1$d-Hausdorff measures of the $\gamma_k$ to that of $\gamma$ (cf. \cite{BGP} Theorem 10.8) that the $d$-length of $\gamma$ is $2\pi$ as claimed. Lastly, using the fact that the $\phi_k$ are homeomorphisms of diminishing distortion it then also follows that the complement $\mathcal U^c$ in $(\mathbb S^2,d)$ also satisfies $(\mathcal U^c, d)\in\mathrm{Alex}^2(1)$, as the non-collapsed limit of the complimentary hemispheres $(U_k^c, d_{g_k})$. 

Therefore, the disk $\overline{\mathcal U}$ and its closed complement in $(\mathbb S^2, d)$ are Alexandrov spaces of $\mathrm{curv}\geqslant 1$, $\mathrm{dim}=2$, and with boundary lengths $2\pi$. By the ``Maximal Volume'' Theorem \ref{GrovePetersen} of Grove-Petersen \cite{GrovePetersenLens}, it follows that each of these disks is an \emph{Alexandrov Lens}: a spherical suspension of a closed interval of length $\leqslant \pi$. That is, the disk $\overline{\mathcal U}$ and its closed complement are isometrically realized as, respectively, 
\[L_\alpha:=\Sigma_1[0,\alpha]:=\left(\left(\left[-\frac{\pi}{2}, \frac{\pi}{2}\right]\times [0,\alpha]\right)/\sim, g=dt^2+\sin^2\left(t+\frac{\pi}{2}\right)d\theta^2\right)\]\[L_\beta:=\Sigma_1[0,\beta]:=\left(\left(\left[-\frac{\pi}{2}, \frac{\pi}{2}\right]\times [0,\beta]\right)/\sim, g=dt^2+\sin^2\left(t+\frac{\pi}{2}\right)d\theta^2\right)\] where $0<\alpha,\beta\leqslant\pi$. Our limit space $(\mathbb S^2, d)$ is therefore isometric to a gluing of these two lenses via \emph{some} isometry their boundaries $\partial L_\alpha\to\partial L_\beta$, and we will work with this explicit representation in the sequel. Note that if $p_\pm$ denote the suspension poles $\overline{\{\pm\pi/2\}\times [0,\alpha]}$ in $L_\alpha$, and $q_\pm$ denote the suspension poles $\overline{\{\pm\pi/2\}\times [0,\beta]}$ in $L_\beta$, then this gluing isometry need not send the pair $p_\pm$ to the pair $q_\pm$\footnote{For the sake of exposition, we recall that Petrunin's Gluing Theorem (see \cite{PetruninApplications}) allows one to glue two spaces in $\mathrm{Alex}^n(\kappa)$ with isometric boundaries (with respect to their induced metrics) together to obtain a new space in $\mathrm{Alex}^n(\kappa)$. In particular any isometry of the boundaries gives rise to a valid gluing map. Of relevance to our work here is the fact that a gluing of two lenses need not send suspension poles of one lens to the suspension poles of the other!}. See Figure \ref{fig:lenses}.

\begin{figure}
     \centering
     \begin{subfigure}[t]{0.3\textwidth}
         \centering
         \includegraphics[width=\textwidth]{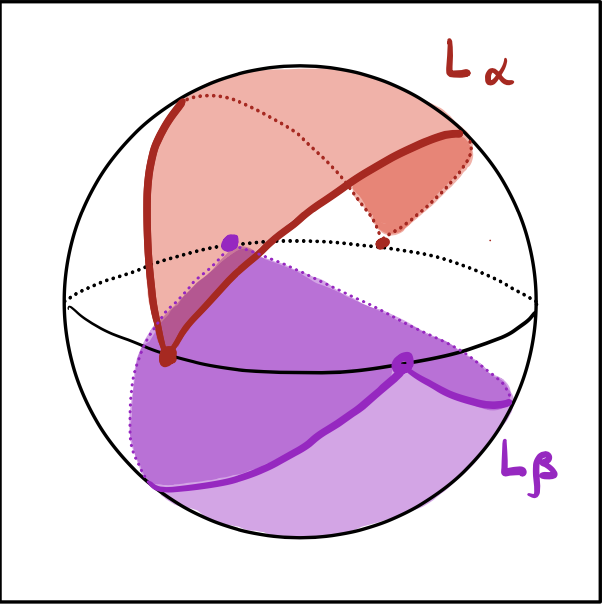}
         \caption{$(\mathbb S^2, d)$ is a gluing of $L_\alpha$ and $L_\beta$.}
         \label{fig:lenses}
     \end{subfigure}\hspace{5em}
     \begin{subfigure}[t]{0.3\textwidth}
         \centering
         \includegraphics[width=\textwidth]{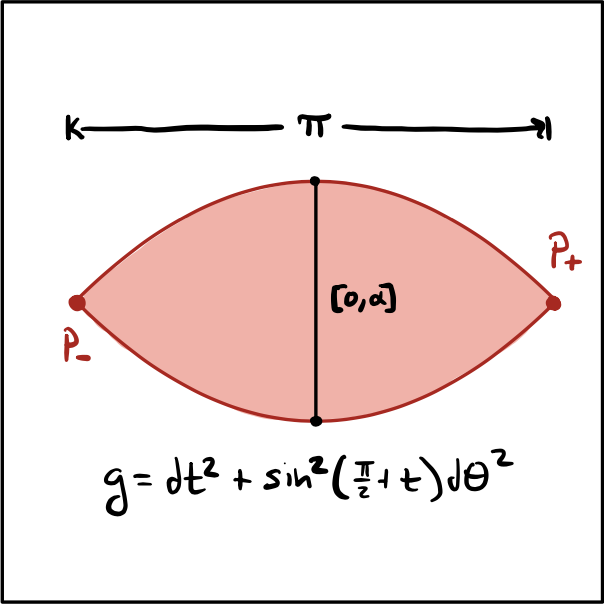}
         \caption{$L_\alpha$ in detail.}
         \label{fig:lensdetail}
     \end{subfigure}
        \caption{Alexandrov Lenses}
        \label{fig:lensesfig}
\end{figure}

Our next step is to show that $\alpha=\beta=\pi$, proving that the limit sphere $\mathbb S^2$ is simply the gluing of two round hemispheres and is therefore itself the round sphere---the desired contradiction. Without loss of generality, suppose for the sake of contradiction that there were an $\eta>0$ such that $\alpha<\pi-2\eta$. We will show that this implies the existence of simple closed geodesics on the manifolds $(M_k, g_k)$ of length strictly less than $\ell_k$ for large enough $k$, an impossibility. 

We focus on the lens $L_\beta=\Sigma_1 [0, \beta]$ to begin the construction of a sweepout of $(\mathbb S^2, d)$, and reparametrize the level sets $\overline{\{t\}\times [0,\beta]}$ by the interval $[0,\pi]$ with constant speed. In case $\beta=\pi$ (and $L_\beta$ is an honest round hemisphere), we are free to first choose the endpoints of the equator $\overline{\{0\}\times [0,\beta]}$ (lying in the boundary of $L_\beta$) to agree with the endpoints of the corresponding equator $\overline{\{0\}\times [0,\alpha]}$ of the other lens $L_\alpha=\Sigma_1 [0,\alpha]$. Similarly, parametrize the level sets $ \overline{\{t\}\times [0,\alpha]}$ in the lens $L_\alpha=\Sigma_1 [0,\alpha]$ by the interval $[0,\pi]$ with constant speed. 

Fix some $0<\tau_0<<1$ very small, and to be determined later. If the $p_\pm$ and $q_\pm$ are distinct pairs of points (i.e. if the gluing map of $L_\alpha$ with $L_\beta$ does not send the suspension poles of $L_\alpha$ to those of $L_\beta$), then we first ensure $\tau_0$ is taken so small that the closed $\tau_0$ d-balls about $p_\pm$ and $q_\pm$ are all disjoint. In the $\beta=\pi$ case we define constant speed sweepout curves $\tilde\sigma(t,\theta)$ by gluing the endpoints of $\overline{\{t\}\times [0,\beta]}$ (which has length $\leqslant\pi$) to those of the $\overline{\{t\}\times [0,\alpha]}$ in the lens $L_\alpha$ (which has length $\leqslant\alpha<\pi-2\eta$) for all $t\in [-\pi/2+\tau_0, \pi/2-\tau_0]$. This is possible because of the length preserving nature of the gluing map between the boundaries of $L_\alpha$ and $L_\beta$, and by our parametrization of $L_\beta$ relative to $L_\alpha$. See Figure \ref{fig:piglue}. The resulting Lipschitz curves $\tilde\sigma(t,\cdot)$ have length $\leqslant 2\pi-2\eta$, and the Lipschitz map $\tilde\sigma: \left[-\pi/2+\tau_0, \pi/2-\tau_0\right]\times\mathbb S^1\to (\mathbb S^2,d)$ gives a ``partial'' sweepout of $(\mathbb S^2, d)$. We also observe, using some elementary spherical geometry, that the curves $\tilde\sigma(\pm\pi/2\mp\tau_0,\cdot)$ remain at least $d$-distance $\Psi(\tau_0)>0$ away from the images of the poles $p_\pm$ in the gluing locus, and stray no farther than $\tau_0$.

In the case where $\beta<\pi$, we proceed similarly, except that we cannot dictate the location of the equator of $L_\beta$ relative to $L_\alpha$. In this case, we build our partial Lipschitz sweepout $\tilde\sigma: \left[-\pi/2+\tau_0, \pi/2-\tau_0\right]\times\mathbb S^1\to (\mathbb S^2,d)$ by taking the $t$-level set of $L_\alpha$ and gluing its endpoints in the shared boundary with $L_\beta$ to those of the unique minimizing round segment in $L_\beta$ between them. See Figure \ref{fig:lessglue}.  Since $\beta<\pi$, these glued in segments vary smoothly with their endpoints, and so we obtain a Lipschitz partial sweepout as before where every curve $\tilde\sigma(t,\cdot)$ has length $\leqslant 2\pi-2\eta$. In this case, we observe using some elementary spherical geometry that the curves $\tilde\sigma(\pm\pi/2\mp\tau_0,\cdot)$ remain at least $d$-distance $\Psi(\tau_0:\beta)>0$ away from $p_\pm$, and stray no farther away than $\tau_0$. 

\begin{figure}
     \centering
     \begin{subfigure}[b]{0.45\textwidth}
         \centering
         \includegraphics[width=\textwidth]{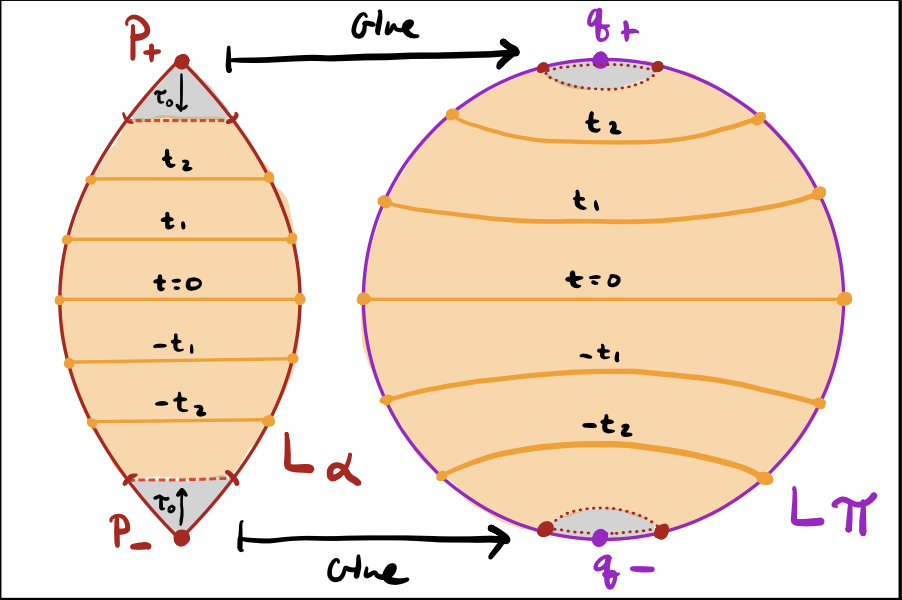}
         \caption{Defining sweepout curves when $\alpha<\pi, \beta=\pi$.}
         \label{fig:piglue}
     \end{subfigure}
     \hspace{2.5em}
     \begin{subfigure}[b]{0.45\textwidth}
         \centering
         \includegraphics[width=\textwidth]{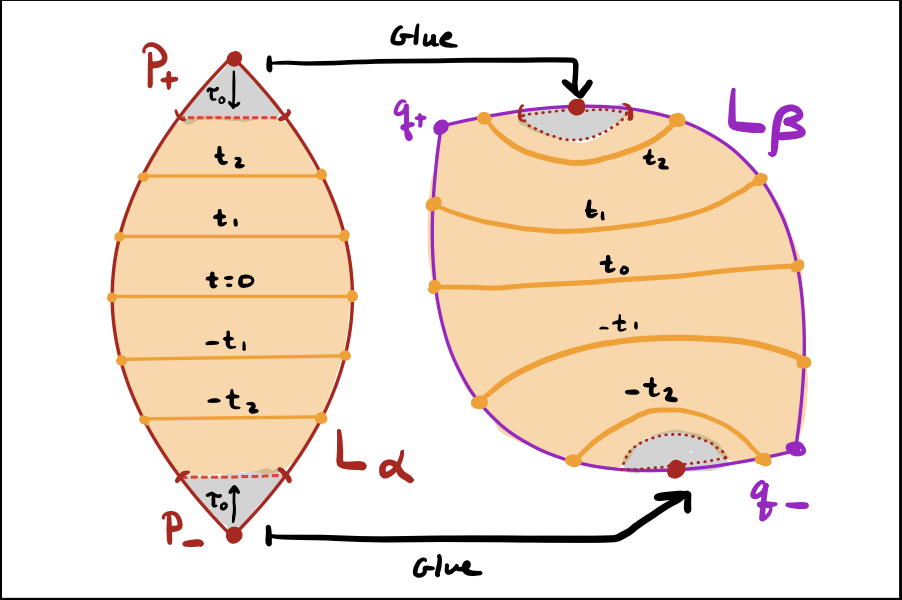}
         \caption{Defining sweepout curves when $\alpha, \beta<\pi$.}
         \label{fig:lessglue}
     \end{subfigure}
        \caption{Two situations for beginning the sweepout construction.}
        \label{fig:sweepoutconstruction}
\end{figure}
  
In whichever case, we intend to pull $\tilde\sigma$ back to the sequence manifolds $(M_k, g_k)$, complete them to full sweepouts, and run Birkhoff's curve shortening process as in \cite{CalabiCao} to obtain a simple closed geodesic of length shorter than $\ell_k$, for some large enough $k$. To do this, we must first find the right maps to pull back by, so that we eventually obtain an admissible sweepout for the min-max theory. We begin by defining the open ``good'' region \[\Omega\coloneqq \left\{x\in(\mathbb S^2,d) \colon d(x, p_{\pm}), d(x, q_{\pm})\geqslant\min\{\Psi(\tau_0:\beta),\Psi(\tau_0)\}\right\}\] where the specific $\Psi(\tau_0:\beta), \Psi(\tau_0)\leqslant\tau_0$ here are those of the previous two paragraphs. 

\begin{lem}\label{ALlemma2}
For any $\varepsilon>0$, for all large enough $k\geqslant 1$ there is an open set $\Omega_k\subset (M_k, g_k)$ such that $\Omega_k$ and $\Omega$, with their \emph{restricted metrics}, are $\varepsilon$-close in the Gromov-Lipschitz sense via a $(1+\varepsilon)$-bi-Lipschitz homeomorphism $f_k\colon \Omega\to f_k(\Omega)=\Omega_k$. In particular, \[(1-\Psi(\varepsilon))\mathrm{Area}_d(\Omega)\leqslant \mathrm{Area}_{g_k}(\Omega_k)\leqslant(1+\Psi(\varepsilon))\mathrm{Area}_d(\Omega).\] Moreover, for any Lipschitz curve $\gamma\subset\Omega$, we have for all $k\geqslant 1$ large enough \[(1-\varepsilon)L_d(\gamma)\leqslant L_{g_k}((f_k)_*(\gamma))\leqslant(1+\varepsilon)L_d(\gamma).\] Lastly, we remark that if $\phi_k:(\mathbb S^2, d)\to (M_k, d_{g_k})$ is an $\varepsilon$-GH approximation, then we can assume without loss of generality that $\phi_k$ and $f_k$ are $\varepsilon$ $d_{g_k}$-uniformly close on $\Omega$. 
\end{lem}

\begin{proof}
Let $\varepsilon>0$. For all large enough $k\geqslant 1$, by Corollary \ref{locYam} we can find homeomorphisms $f_k\colon\Omega\to f_k(\Omega):=\Omega_k\subset M_k$ which are $\varepsilon$-almost isometries: \[\left\vert\frac{d_{g_k}(f_k(x),f_k(y))}{d(x,y)}-1\right\vert<\varepsilon\quad\text{for all } x, y\in\Omega.\] Indeed, the $\delta$-strain radius of any point $p\in\Omega$ is at least $\min\{\Psi(\tau_0:\beta), \Psi(\tau_0)\}>0$ by construction, for any $\delta>0$. Thus, $(\Omega_k, d_{g_k}\vert_{\Omega_k\times\Omega_k})$ and $(\Omega, d\vert_{\Omega\times\Omega})$ are $\varepsilon$-close in the Gromov-Lipschitz sense. The first claim now follows from the definition of the Hausdorff measures associated to $d_{g_k}$ and $d$, and the second claim follows from the definition of the length of a curve in a metric space. The final remark follows from the construction of $f_k$ starting from such a Gromov-Hausdorff approximation $\phi_k$. 
\end{proof}

Returning to the main argument, we may now pull the partial sweepout $\tilde\sigma$ that was just constructed on $(\mathbb S^2, d)$ back to the sequence spaces $(M_k, g_k)$ via the maps $f_k$. For $\tau_0>0$  still to be further determined and $0<\varepsilon<(2\pi-\eta)/(2\pi-2\eta)-1$ (which we will also take even smaller later on), for each $k\geqslant 1$ large enough we fix $f_k$, $\Omega$, $\Omega_k$ as in Lemma \ref{ALlemma2} and define the partial sweepout maps $\tilde\sigma_k$ (defined on some domain $\mathrm{Dom}(\tilde\sigma_k)\subset [-\pi/2,\pi/2]\times\mathbb S^1$) by \[\tilde\sigma_k(t,\theta)=f_k\circ\sigma(t,\theta): \mathrm{Dom}(\tilde\sigma_k)\to (M_k, g_k).\] Since each $f_k$ is bi-Lipschitz, the $\tilde\sigma_k$ are themselves Lipschitz. For each fixed $t$ where $\tilde\sigma_k(t,\cdot)$ is defined on a nonempty subset of $\mathbb S^1$, the (possibly disconnected) curve $(\tilde\sigma_k)_t=\tilde\sigma_k(t,\cdot)$ in $M_k$ has length \[L_{g_k}((\tilde\sigma_k)_t)\leqslant (1+\varepsilon)L_d(\sigma_t)\leqslant (1+\varepsilon)(2\pi-2\eta)<2\pi-\eta.\]

Of course, the image of $\tilde\sigma_k$ misses two or four small Lipschitz disks in $M_k$, and some curves $(\tilde\sigma_k)_t$ might not be closed with up to two disconnected components. To obtain a full sweepout of $M_k$, we must fill in these missing portions. See Figure \ref{fig:partialsweepout}: 
\begin{figure}
    \centering
    \includegraphics[width=0.65\linewidth]{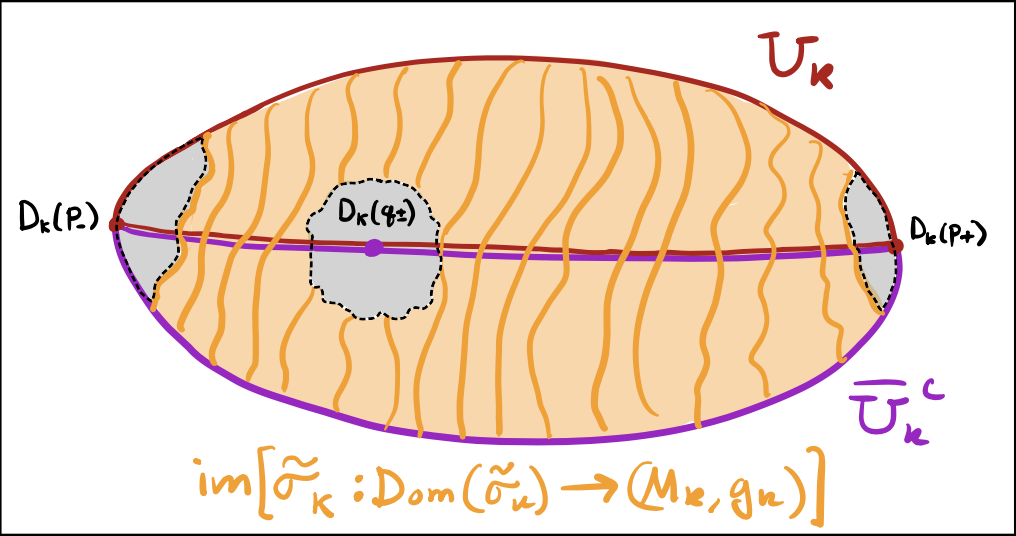}
    \caption{The partial sweepout $\tilde\sigma_k$.}
    \label{fig:partialsweepout}
\end{figure} By the last estimate in Lemma \ref{ALlemma2}, relative to the $g_k$ metric structures the boundaries of these disks, namely the boundary components of $\Omega_k$, are Lipschitz curves. We may thus apply Theorem \ref{LNR} to sweep out these remaining disks, and glue these respective homotopies into $\tilde\sigma_k$ to obtain a full sweepout $\sigma_k:\left[-\pi/2, \pi/2\right]\times\mathbb S^1\to (M_k, g_k)$.

To wit, first let $D_k(p_\pm)$ be the smallest disk in $M_k$ bounded by the curve  $(\tilde\sigma_k)_{t_\pm}$ for $t_\pm=\pm\pi/2\mp\tau_0$. By Theorem \ref{LNR}, the $D_k(p_\pm)$ may be swept out by Lipschitz homotopies $\sigma_{k,\pm}(t,\theta)$ as in Figure \ref{fig:Dp}, parameterized to have the domains $\left[-\pi/2,-\pi/2+\tau_0\right]\times\mathbb S^1$ and  $\left[\pi/2-\tau_0,\pi/2\right]\times\mathbb S^1$, respectively, and which obey 

\begin{enumerate}[label=\textbf{(\Alph*)}]
\setlength\itemsep{.5em}
    \item $\sigma_{k,\pm}(t_\pm,\theta)=\tilde\sigma_k(t_\pm,\theta)$ for every $\theta\in[0,2\pi]$;
    \item $\sigma_{k,\pm}(\pm\pi/2,\theta)\equiv q_{k,\pm}$ for any fixed $q_{k,\pm}\in\partial D_k(p_\pm)$ and for every $\theta\in[0,2\pi]$;
    \item \label{(C)} $L_{g_k}(\sigma_{k,\pm}(t,\cdot))\leqslant 2L_{g_k}(\partial D_k(p_\pm))+686\sqrt{\mathrm{Area}_{g_k}(D_k(p_\pm))}+2\mathrm{diam}_{g_k}(D_k(p_\pm)).$
\end{enumerate} Let us consider estimate \ref{(C)} further. We claim that by choosing $\tau_0, \varepsilon>0$ small enough and $k\geqslant 1$ large enough, the length of any curve in these homotopies is strictly bounded above by $2\pi-\eta$. In the first term of \ref{(C)}, we use Lemma \ref{ALlemma2} to estimate \[L_{g_k}(\partial D_k(p_\pm))\leqslant (1+\varepsilon)L_d(\tilde\sigma_{t_\pm})\leqslant (1+\varepsilon)\Psi(\tau_0: \beta).\] For the second term of \ref{(C)}, by weak convergence of the Hausdorff measures of the $(M_k, g_k)$ to that of $(\mathbb S^2, d)$ (see Theorem 10.8 in \cite{BGP}) and Lemma \ref{ALlemma2}, we have 
\begin{align*}
\mathrm{Area}_{g_k}(D_k(p_\pm))&\leqslant\mathrm{Area}_{g_k}(M_k)-\mathrm{Area}_{g_k}(\Omega_k)\\ 
&\leqslant \mathrm{Area}_d(\mathbb S^2)+\Psi(k^{-1})-(1-\Psi(\varepsilon))\mathrm{Area}_d(\Omega)\\ 
&=\mathrm{Area}_d(\mathbb S^2)+\Psi(k^{-1})-(1-\Psi(\varepsilon))(\mathrm{Area}_d(\mathbb S^2)-\Psi(\tau_0:\beta))\\
&=\Psi(k^{-1})+\Psi(\varepsilon)+\Psi(\tau_0:\beta).\end{align*}
To handle the last term of \ref{(C)}, we first recall from the last part of Lemma \ref{ALlemma2} that the map $f_k$ is $\varepsilon$-close in the $d_{g_k}$ uniform sense on $\Omega$ to a \emph{global} $\varepsilon$-GH approximation $\phi_k$ between $(\mathbb S^2, d)$ and $(M_k, g_k)$. Let $x\in D_k(p_-)$, say, and $\xi\in\mathbb S^2$ such that $\phi_k(\xi)=x$. Let also $p_{k,-}=\phi_k(p_-)$. Since $\mathrm{dis}(\phi_k)<\varepsilon$, we have that \[d_{g_k}(x,p_{k,-})\leqslant d(\xi, p_-)+\varepsilon.\] Suppose for the sake of contraction that $d(\xi, p_-)>\tau_0+2\varepsilon$. Then by definition $\xi\in\mathrm{int}~\Omega$, and for arbitrary $\zeta\in\partial\Omega$, $d(\xi,\zeta)\geqslant d(\xi, \partial\Omega)>2\varepsilon.$ By Lemma \ref{ALlemma2}, we thus obtain for any $\zeta\in\partial\Omega$ \[d_{g_k}(f_k(\xi), f_k(\zeta))\geqslant (1-\varepsilon)d(\xi,\zeta)>2\varepsilon(1-\varepsilon)\] which shows that $d_{g_k}(f_k(\xi), \partial\Omega_k)\geqslant 2\varepsilon(1-\varepsilon)$. On the other hand, for small $\varepsilon>0$ \[d_{g_k}(x,f_k(\xi))=d_{g_k}(\phi_k(\xi),f_k(\xi))\leqslant\varepsilon<2\varepsilon(1-\varepsilon),\] so we would obtain the contradiction that $x\in\Omega_k\subset M_k\setminus D_k(p_-)$. Therefore, we conclude that $d(\xi, p_-)\leqslant\tau_0+2\varepsilon$ and thus \[d_{g_k}(x,p_{k,-})\leqslant \tau_0+3\varepsilon.\] For any $x,y\in D_k(p_-)$ we then obtain \[d_{g_k}(x,y)\leqslant 2\tau_0+6\varepsilon,\] and analogously for $D_k(p_+)$. 

Altogether, then, we find that estimate \ref{(C)} becomes \[L_{g_k}(\sigma_{k,\pm}(t,\cdot))\leqslant \Psi(\tau_0, \varepsilon, k^{-1}: \beta),\] and we extend the Lipschitz partial sweepout $\tilde\sigma_k$ by gluing in these nullhomotopies of the $D_k(p_\pm)$.

\begin{figure}
    \centering
\includegraphics[width=0.35\linewidth]{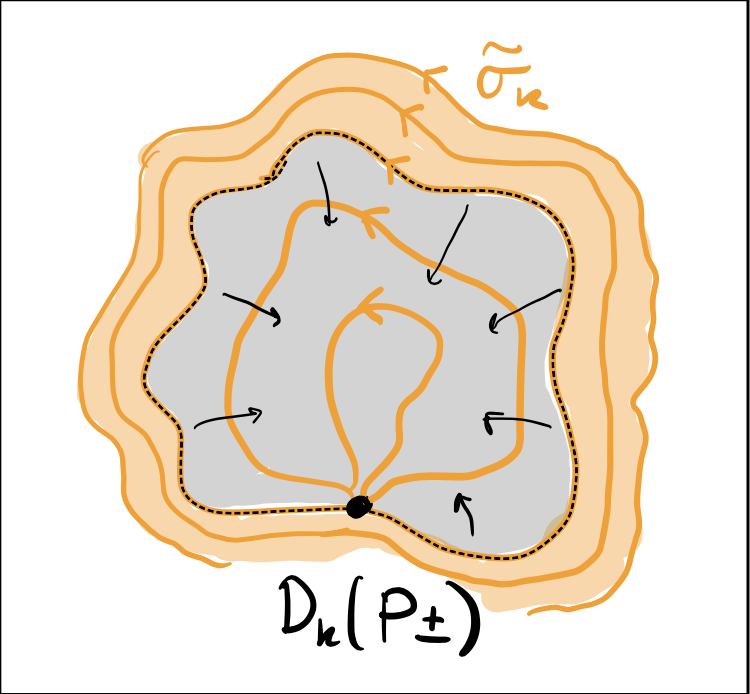}
    \caption{Sweeping out $D_k(p_\pm)$.}
    \label{fig:Dp}
\end{figure}

Consider now the disks $D_k(q_\pm)$, and for the sake of illustration fix attention on $D_k(q_+)$. In this case, the situation is as in \textbf{(A)-(F)} of Figure \ref{fig:Dq}. \textbf{(A), (F):} Let $t_0$ be the smallest $t$ such that $(\tilde\sigma_k)_t$ contacts $\partial D_k(q_+)$, and let $t_1$ be the largest. \textbf{(B), (C):} For $t\in [t_0, t_1]$, the curves $(\tilde\sigma_k)_t$ are cut by $\partial D_k(q_+)$, and we begin filling in the missing portions by gluing in the arcs of the (Lipschitz) boundary of $\partial D_k(q_+)$ between the subsequent endpoints of $(\tilde\sigma_k)_t$ on $\partial D_k(q_+)$ in such a way that as $t$ goes from $t_0$ to $t_1$, the new curves move continuously. \textbf{(D):} When we finally reach $t_1$, we glue in the entirety of $\partial D_k(q_+)$, and then \textbf{(E)} use Theorem \ref{LNR} to homotope this loop, while holding the original portion $(\tilde\sigma_k)_{t_1}$ fixed, to the point of tangency where they were glued together (of course, this will require us to eventually reparametrize the ``time'' parameter $t$ of the sweepouts). By completely analogous estimates as in the $D_k(p_\pm)$ fill-ins, the added length to the original partial sweepout curves can be made $\leqslant \Psi(\tau_0, \varepsilon, k^{-1})$, and the resulting extension of the sweepout maintains its Lipschitz regularity. This completes the fill-in proceedure for the damage left by $D_k(q_+)$, and we carry out a similar fill-in for $D_k(q_-)$. 

\begin{figure}
    \centering
    \includegraphics[width=0.65\linewidth]{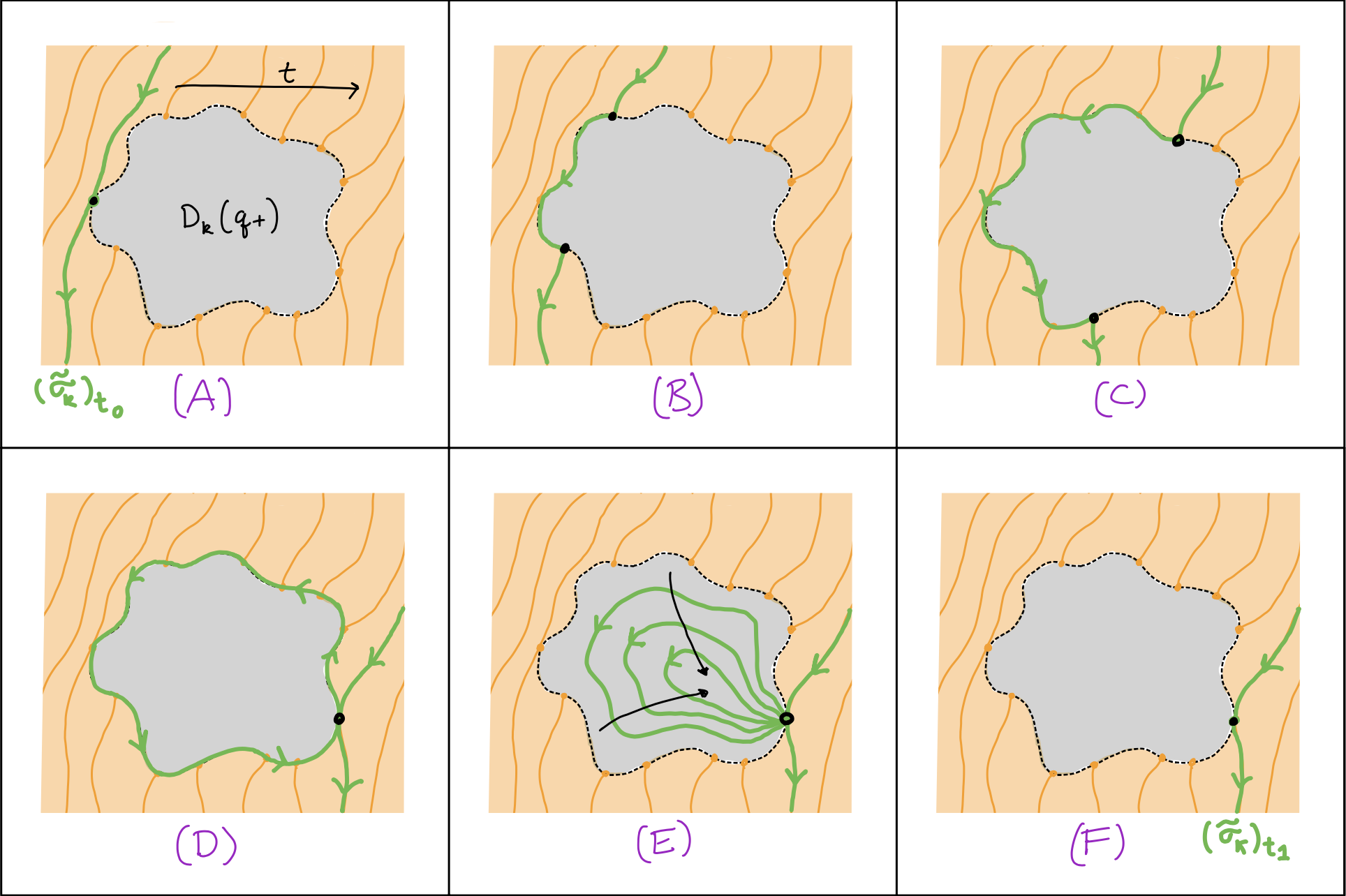}
    \caption{Patching the sweepout in $D_k(q_+)$.}
    \label{fig:Dq}
\end{figure}

We finally reparametrize the time parameter $t$ to lie in $[-\pi/2,\pi/2]$, and take the parameters $\tau_0, \varepsilon>0$ so small, and $k\geqslant 1$ so large, that for all $t$ the sweepout curves $(\sigma_k)_t$ satisfy \[L_{g_k}((\sigma_k)_t)<2\pi-\eta.\] At this stage, we consider the parameters as all fixed, and the Lipschitz maps  $\sigma_k: [-\pi/2,\pi/2]\times\mathbb S^1\to(M_k, g_k)$ extending the partial sweepouts $\tilde\sigma_k$ as above constructed for all large enough $k\geqslant 1$.  Clearly, $\deg_2\sigma_k\neq 0$. Thinking of the induced maps \[ \left[-\frac{\pi}{2}, \frac{\pi}{2}\right]\ni t\mapsto \llbracket(\sigma_k)_t( \cdot)\rrbracket\in \mathcal Z_1((M_k, g_k), \mathbb Z),\] we see that $\llbracket(\sigma_k)_{\pm\pi/2}\rrbracket$ lie in the zero class (in particular $(\sigma_k)_{\pm\pi/2}$ are point curves). The induced map $t\mapsto\llbracket(\sigma_k)_t\rrbracket$ into integral one-cycles is also continuous with respect to the weak topology on currents. To see this, let $s\nearrow t$ in $[-\pi/2,\pi/2]$, and define \[\Sigma_k(s,t)=\sigma_k([s,t]\times\mathbb S^1).\] Note that for every $s<t$, $\llbracket [s,t]\times\mathbb S^1\rrbracket\in\mathcal \mathcal I_2([-\pi/2,\pi/2]\times\mathbb S^1)$, and since the $\sigma_k$ are Lipschitz, $\llbracket\Sigma_k(s,t)\rrbracket=(\sigma_k)_{\#}\llbracket [s,t]\times\mathbb S^1\rrbracket\in\mathcal \mathcal{I}_2(M_k)$ with, crucially, \[\llbracket(\sigma_k)_t\rrbracket-\llbracket(\sigma_k)_s\rrbracket=\pm\partial\llbracket\widetilde{\Sigma_k}(s,t)\rrbracket\] for $\widetilde{\Sigma_k}(s,t)$ some Lipschitz regular subregion of $\Sigma_k(s,t)$. Indeed, before filling in the partial sweepout we would have $\widetilde\Sigma_k(s,t)\equiv\Sigma_k(s,t)$, but the fill-in procedure could plausibly cause the sweepout map $\sigma_k$ to lose injectivity. As $s\nearrow t$, $\mathrm{Area}_{dx^2\oplus d_{\mathbb S^1}}([s,t]\times\mathbb S^1)\to 0$, and thus also $\mathrm{Area}_{g_k}(\widetilde{\Sigma_k}(s,t))\leqslant\mathrm{Area}_{g_k}(\Sigma_k(s,t))\to 0$ (again using that $\sigma_k$ is Lipschitz). But then $\llbracket(\sigma_k)_s\rrbracket\to\llbracket(\sigma_k)_t\rrbracket$ in the weak topology, since for any fixed $\theta\in\Omega^1_c(M_k)$, \begin{align*}\llbracket (\sigma_k)_t\rrbracket(\theta)-\llbracket (\sigma_k)_s\rrbracket(\theta)&=\pm\partial\llbracket\widetilde{\Sigma_k}(s,t)\rrbracket(\theta)=\pm\llbracket\widetilde{\Sigma_k}(s,t)\rrbracket(\mathrm{d}\theta)=\pm\int_{\widetilde{\Sigma_k}(s,t)}\mathrm{d}\theta\to 0.
\end{align*} An identical argument handles the situation where $s\searrow t$ in $[-\pi/2,\pi/2]$. Therefore, $\sigma_k$ is a continuous and non-contractible loop in $\pi_1(\mathcal Z_1((M_k, g_k), \mathbb Z), \llbracket 0\rrbracket)$ which starts and ends at $\llbracket 0\rrbracket$. In short, $\sigma_k$ induces an admissible sweepout of $(M_k, g_k)$ by integral one-currents. 

Now, fix a $k_0\geqslant 1$ so large that $\ell_{k_0}\geqslant 2\pi-\eta$. By construction the longest curve in the sweepout $\sigma_{k_0}$ is of length $<2\pi-\eta$, so by the min-max theory of Almgren-Pitts and Calabi-Cao (see Theorem \ref{AlmgrenPitts}) there exists a simple closed geodesic in $(M_{k_0}, g_{k_0})$ of length less than $2\pi-\eta<\ell_{k_0}$, a contradiction. 

Therefore, we conclude that the angle $\alpha$ of the lens $L_\alpha$ must be $\pi$, and then the same argument applied starting with the lens $L_\beta$ shows that $\beta=\pi$ as well. Thus, the potential singular set of poles is empty and the limit metric is the smooth round metric on $\mathbb S^2$. This proves that the $(M_k,g_k)$ subconverge in the Gromov-Hausdorff sense to $(\mathbb S^2, g_{rd})$. Since the singular set of the limit is now known to be empty (in particular, $\delta\mathrm{str.rad}(\mathbb S^2, g_{rd})=\pi$ for every $\delta>0$) , we can upgrade the Gromov-Hausdorff convergence just obtained to $C^0$-Cheeger-Gromov convergence by Theorem \ref{Yam} and Proposition \ref{GHtoC0}. This forces the desired contradiction and proves the main theorem.
\end{proof}

\section{Funding} The first author was partially supported by NSF grant DMS-1910496

\bibliographystyle{alpha}
\bibliography{biblio1.bib}
\nopagebreak
\end{document}